\documentclass[a4paper]{amsart}
\usepackage[english]{babel}
\usepackage{amsmath, amssymb, amsthm, amscd}
\usepackage{enumerate}
\usepackage{palatino}
\usepackage{mathpazo}
\usepackage{paralist}
\usepackage[a4paper]{geometry}
\usepackage{url}
\usepackage{hyperref}
\usepackage[all]{xy}

\DeclareMathOperator{\im}{im}

\newcommand{\Fix}{\mathrm{Fix\;}}

\newcommand{\NN}{\mathbb{N}}
\newcommand{\ZZ}{\mathbb{Z}}
\newcommand{\QQ}{\mathbb{Q}}
\newcommand{\RR}{\mathbb{R}}
\newcommand{\CC}{\mathbb{C}}
\newcommand{\CP}{\mathbb{C}P}
\newcommand{\HP}{\mathbb{H}P}
\newcommand{\KP}{\mathbb{K}P}
\newcommand{\K}{\mathbb{K}}

\newcommand{\DD}{\mathcal{D}}
\newcommand{\Dop}{\mathcal{D}^{\mathrm{op}}}

\newcommand{\ind}{\mathrm{ind}}

\newcommand{\D}{\mathcal{D}}

\newcommand{\wgt}{\mathrm{wgt}}

\newcommand{\Hom}{\mathrm{Hom}}

\newcommand{\F}{\mathbb{F}}

\DeclareMathOperator{\SC}{\mathsf{SC}}

\DeclareMathOperator{\secat}{\mathsf{secat}}

\newcommand{\gFS}{g_{\mathrm{FS}}}

\theoremstyle{plain}
\newtheorem{theorem}{Theorem}[section]
\newtheorem{prop}[theorem]{Proposition}
\newtheorem{lemma}[theorem]{Lemma}
\newtheorem{cor}[theorem]{Corollary}  

\newtheorem*{theorem*}{Theorem}

\theoremstyle{definition}
\newtheorem{definition}[theorem]{Definition}

\theoremstyle{remark}
\newtheorem{remark}[theorem]{Remark}

\numberwithin{equation}{section}

\begin{document}
\setlength{\parindent}{0cm}

\title{Existence results for closed Finsler geodesics via spherical complexities}
\author{Stephan Mescher}
 \address{Mathematisches Institut \\ Universit\"at Leipzig \\ Augustusplatz 10 \\ 04109 Leipzig \\ Germany}
\email{mescher@math.uni-leipzig.de}
\date{\today}
\maketitle
\begin{abstract}
We apply topological methods and a Lusternik-Schnirelmann-type approach to prove existence results for closed geodesics of Finsler metrics on spheres and projective spaces. The main tool in the proofs are spherical complexities, which have been introduced in earlier work of the author. Using them, we show how pinching conditions and inequalities between a Finsler metric and a globally symmetric metric yield the existence of multiple closed geodesics as well as upper bounds on their lengths. 
\end{abstract}

\section{Introduction}

In \cite{SpherComp}, the author has introduced integer-valued homotopy-invariants based on spaces of continuous maps from spheres to topological spaces. Given a closed manifold $M$, these invariants can be used to estimate numbers of orbits of critical points of $G$-invariant differentiable functions on Hilbert submanifolds of $C^0(S^n,M)$, where $G$ is a subgroup of $O(n+1)$ and where we consider the $O(n+1)$-action by reparametrization, i.e. the one induced by the standard $O(n+1)$-action on $S^n$. For $n=1$, this leads to estimates of critical orbits of $O(2)$-invariant or $SO(2)$-invariant functions on Hilbert manifolds of free loops in $M$. \bigskip 

A typical situation in which this method applies is the study of energy functionals of Riemannian or Finsler metrics on $H^1(S^1,M)$, the Hilbert manifold of free loops in $M$ that is locally modelled on the Sobolev space $H^1(S^1,\RR^{\dim M})=W^{1,2}(S^1,\RR^{\dim M})$. The critical points of such functionals are precisely the closed geodesics of the metric under consideration and methods from Morse theory and Lusternik-Schnirelmann theory have been very successfully implemented to derive existence results for closed geodesics in the previous decades, we refer e.g. to \cite{OanceaGeod} for an overview. Here, we only want to present certain recent results that are close to or intersect with the results shown in this note. In the following, we will always let $\lambda$ denote the reversibility of the Finsler metric under consideration, see \cite{RadeSphere} or \cite{RadeNonrev} for the definition of reversibility. For $x\in \RR$ we further let $\lceil x \rceil$ denote the smallest integer that is bigger than or equal to $x$.

\begin{itemize}
\item In \cite[Theorem 3]{RadeExist}, H.-B. Rademacher has shown that any Finsler metric on $S^n$ whose flag curvature satisfies $\frac{\lambda^2}{(1+\lambda)^2}< K \leq 1$ admits $\left\lceil \frac{n}{2}-1\right\rceil$ distinct prime closed geodesics whose lengths are less than $2n\pi$.
\item Rademacher has further shown that any Finsler metric on $S^{2n}$,  $n\geq 3$, whose flag curvature satisfies $\frac{1}{(n-1)^2} < K \leq 1$ admits two distinct prime closed geodesics of length less than $2n\pi$.
\item In \cite{LongDuan} and \cite{DuanLong}, Y. Long and H. Duan have proven that every Finsler metric on $S^3$ and $S^4$ admits two distinct prime closed geodesics. 
\item Duan has shown in \cite{DuanNonhyp} that any Finsler metric on $S^n$ whose flag curvature satisfies $\frac{\lambda^2}{(1+\lambda)^2} < K \leq 1$ admits three distinct prime closed geodesics.
\end{itemize}

While Rademacher is using methods of comparison theory and the Fadell-Rabinowitz index to show his results, Duan and Long apply a classification method of closed geodesics by their linearized Poincar\'{e} maps and Morse-theoretic arguments for bumpy metrics. Note that while Duan and Long obtain very general results, their approach does not yield any upper bound on the length of the closed geodesics whose existence they have shown. \bigskip

We want to derive similar results using more topological methods, namely using spherical complexities. The author has shown in \cite{SpherComp} that the invariants introduced therein satisfy a Lusternik-Schnirelmann-type theory and established a relationship between these invariants and critical points of differentiable functions, which we want to summarize for the case of free loop spaces. Given a closed manifold $M$ let $\Lambda_0M= \{\gamma \in C^0(S^1,M)\; | \; \gamma \text{ is nullhomotopic} \}$. For $A \subset \Lambda_0M$ we let $\SC_M(A)$ denote the smallest number $r \in \NN\cup \{+\infty\}$, for which $A$ can be covered by open subsets $U_1,\dots,U_r\subset \Lambda_0M$ such that for each $j \in \{1,2,\dots,r\}$ there exists a continuous map $s:U_j \to C^0(B^2,M)$ with $s(\gamma)|_{S^1}=\gamma$ for each $\gamma \in U_j$. Here, $B^2$ denotes the closed unit disk in $\RR^2$.

Let $f:\Lambda^1M \to \RR$ be a $G$-invariant $C^{1,1}$-function that is bounded from below, where $\Lambda^1M:=H^1(S^1,M)\cap \Lambda_0M$ and $G\subset O(2)$ is a closed subgroup. Denote its sublevel sets by $f^a := f^{-1}((-\infty,a])$.	 In \cite{SpherComp}, it is shown that under mild additional assumptions, the number $\SC_M(f^a)-1$ provides a lower bound on the number of non-constant critical $G$-orbits in $f^a$. Relating the invariants $\SC_M$ to the topological notion of sectional categories of fibrations, one can establish lower bound on the numbers $\SC_M(f^a)$ in terms of  cohomology rings of the free loop space of $M$.

In this note, we apply this topological approach and results from \cite{SpherComp} to derive existence results for closed geodesics of Finsler metrics of positive flag curvature on spheres and projective spaces. A first result on the existence of closed Finsler geodesics using spherical complexities was shown by the author in \cite[Theorem 7.4]{SpherComp}. Foundations on the notion of Finsler metrics and their geodesics can be found in \cite{Shen} or \cite{RadeNonrev}. \bigskip

We want to give an overview over the results of this article. Let $M$ be a closed manifold and $F:TM \to [0,+\infty)$ be a Finsler metric on $M$ of reversibility $\lambda$ whose flag curvature satisfies $0<K\leq 1$. Denote the length of its shortest non-trivial closed geodesic by $\ell_F$. Then:
\begin{itemize}
\item for $M=S^{2n}$, $n\geq 2$, there will be two distinct closed geodesics of length less than $2\ell_F$ if one of the following holds: 
\begin{itemize}
\item[$\cdot$] $ F \leq \frac{1+\lambda}\lambda \sqrt{g_1}$, where $g_1$ is the round metric of curvature $1$ (Theorem \ref{TheoremGeodSpheresEven}),
\item[$\cdot$] $ K > \frac{9\lambda^2}{4(1+\lambda)^2}$ (Theorem \ref{TheoremGeodSpheresEvenPinch}),
\end{itemize}
\item for $M=S^{2n+1}$, $n \in \NN$, and given $k,m \in \NN$ there will be $\left\lceil \frac{2m}k \right\rceil$ distinct closed geodesics of length less than $(k+1)\ell_F$ if one of the following holds:
\begin{itemize}
\item[$\cdot$]  $\displaystyle K > \frac{\lambda^2}{(1+\lambda)^2} \quad \text{and} \quad F < \frac{(k+1)(1+\lambda)}{m\lambda} \sqrt{g_1}$ \enskip (Theorem \ref{TheoremGeodSpheresOdd}),
\item[$\cdot$] $\displaystyle K> \frac{(2m+1)^2\lambda^2}{(k+1)^2(1+\lambda)^2}$ \enskip (Theorem \ref{TheoremGeodSpheresOddPinch}),
\end{itemize}
\item for $M=\CP^n$ or $M=\HP^n$, $n \geq 3$, there will be two distinct closed geodesics of length less than $2\ell_F$ if $ F < \frac{1+\lambda}{\lambda}\sqrt{g_1}$, where $g_1$ is the globally symmetric Riemannian metric whose prime closed geodesics have length $2\pi$ (Theorem \ref{TheoremKPn}).
\end{itemize}

All of these results are summarized in the following table:
\begin{center}
\begin{tabular}{c|c|c|c|c}
M & number & $<$ length & flag curvature & metric estimate \\
\hline
$S^{2n},\: n \geq 2$ & $2$ & $2\ell_F$ & $0 < K \leq 1$ & $F \leq \frac{1+\lambda}\lambda \sqrt{g_1}$ \\
$S^{2n},\: n \geq 2$ & $2$ & $2\ell_F$ & $\frac{9\lambda^2}{4(1+\lambda)^2}<K\leq 1$ &  -  \\		
$S^{2n+1}$ & $\left\lceil\frac{2m}{k}\right\rceil$ & $(k+1)\ell_F$ & $\frac{\lambda^2}{(1+\lambda)^2}<K\leq 1$ &  $F \leq \frac{(k+1)(1+\lambda)}{m\lambda} \sqrt{g_1}$  \\	
$S^{2n+1}$ & $\left\lceil\frac{2m}{k}\right\rceil$ & $(k+1)\ell_F$ & $\frac{(2m+1)^2\lambda^2}{(k+1)^2(1+\lambda)^2}<K\leq 1$ & -  \\	
$\CP^n, \; \HP^n, \;n \geq 3$ & $2$ & $2 \ell_F$ & $0 < K \leq 1$ & $F \leq \frac{1+\lambda}{\lambda}\sqrt{g_1}$  
\end{tabular}
\end{center}

\bigskip 

In section 2 we recall the constructions and main results from \cite{SpherComp} for the case of free loop spaces. These results will be applied to Finsler energy functionals on even- and odd-dimensional spheres, resp., in sections 3 and 4. Section 5 considers Finsler metrics on complex and quaternionic projective spaces.

\section*{Acknowledgements}

The author thanks Hans-Bert Rademacher for kindly sharing his expertise on Finsler geodesics and for helpful comments on an earlier draft of the manuscript. He also thanks Philip Kupper for helpful conversations about free loop spaces and the anonymous referee for helpful comments improving the quality of the article.

\section{Spherical complexities on loop spaces}

We recall some definitions and results about spherical complexities from \cite{SpherComp}. We will restrict to the setting of free loop spaces and slightly simplify the notation, e.g. the numbers $\SC_X(A)$ appearing below are denoted by $\SC_{1,X}(A)$ in \cite{SpherComp}.

Given a topological space $X$ we let $\Lambda X := C^0(S^1,X)$ denote its free loop space and let 
$$\Lambda_0X := \{ \gamma \in \Lambda X \ | \ \gamma \text{ is nullhomotopic}\}$$ 
denote the set of contractible loops in $X$. We further put $BX := C^0(B^2,X)$ where $B^2$ denotes the closed unit disk in $\RR^2$. 

\begin{definition}
\begin{enumerate}[(1)]
\item Let $A \subset \Lambda_0X$. A \emph{loop filling over $A$} is a continuous map $s:A \to BX$ with $s(\gamma)|_{S^1} = \gamma$ for all $\gamma\in A$. We call $A$ a \emph{loop filling domain} if there exists a loop filling over $A$, let $\DD(X)$ be the set of all loop filling domains in $\Lambda_0 X$ and put 
$$\Dop(X) := \{ U \in \D(X) \ | \ U \text{ is open in } \Lambda_0X\}.$$
\item Given $A \subset \Lambda_0X$ we put
$$\SC_X(A) := \inf \Big\{ r \in \NN \ \Big| \ \exists U_1,\dots,U_r \in \Dop(X) \text{ with } A \subset \bigcup_{i=1}^r U_i\Big\} \in \NN \cup \{+\infty\}.$$
\end{enumerate}
\end{definition}
\begin{remark}
The numbers $\SC_X(A)$ are related to the notion of \emph{sectional category} (or \emph{Schwarz genus}) of a fibration, see e.g. \cite[Section 9.3]{CLOT}. Let $r: BX \to \Lambda_0X$, $\gamma \mapsto \gamma|_{S^1}$. Since the inclusion $i: S^1 \hookrightarrow B^2$ is a cofibration, $r$ is a fibration.  In terms of sectional categories of fibrations, it was shown in \cite[Proposition 2.2]{SpherComp} that
\begin{equation}
\label{EqSCrelsecat}
\SC_X(A) \geq \secat\Big(i_A^*r: i_A^*BX \to A \Big), 
\end{equation}
for all $A\subset \Lambda_0X$, where $i_A:A \hookrightarrow \Lambda_0X$ denotes the respective inclusion. This relation to sectional categories can be exploited to derive lower bounds for them in terms of cup lengths of certain ideals of the cohomology rings $H^*(\Lambda_0X;R)$, see \cite[Proposition 2.4]{SpherComp}.
\end{remark}

In \cite{SpherComp}, the author has carried out a Lusternik-Schnirelmann-type theory for the numbers $\SC_X$ and their higher-dimensional generalizations. We are going to give a brief account of the results that are relevant for the study of closed geodesics.

We consider the $O(2)$-action on $\Lambda_0X$  given by reparametrization of loops and its restrictions to subgroups of $O(2)$. Given a topological space $X$, we further let 
$$c_1: X \to \Lambda_0X, \quad (c_1(x))(t)=x\quad \forall x \in X, \ t \in S^1,$$
 i.e. $c_1$ is the inclusion of constant loops. 
Given any function $f:Y \to \RR$ on a topological space $Y$ we denote its closed and open sublevel sets, resp., by 
$$f^a := f^{-1}((-\infty,a]), \qquad f^{<a} := f^{-1}((-\infty,a)), \qquad \forall a \in \RR.$$

\begin{theorem}[{\cite[Theorem 2.19]{SpherComp}}]
\label{TheoremLS}
Let $X$ be a metrizable absolute neighborhood retract, $G\subset O(2)$ be a closed subgroup, $A\subset \Lambda_0X$ be $G$-invariant with $c_1(X)\subset A$ and $f: A \to \RR$ be continuous and $G$-invariant. Let $\varphi: A \to A$ be a homotopy equivalence, such that $\Fix \varphi\subset A$ is $G$-invariant. If 
\begin{itemize}
\item $f$ is bounded from below, 
\item $f$ is constant on $c_1(X)$,
\item $(f,\varphi)$ satisfies the following conditions: 
\begin{enumerate}[({D}1)]
\item $f(\varphi(x))<f(x)$ for all $x \in A\setminus \Fix \varphi$, 
\item If $f$ is bounded on $B \subset A$, but $\{f(y)-f(\varphi(y)) \ | \ y \in B\}$ is not bounded away from zero, then $\bar{B} \cap \Fix \varphi \neq \varnothing$,
\end{enumerate}
\item $f(\Fix \varphi)$ is isolated in $\RR$,
\end{itemize}
then $$\nu(f,\varphi,a) \geq \SC_X(f^{a})-1$$
for each $a\in \RR$, where $\nu(f,\varphi,a)$ denotes the number of $G$-orbits in $\Fix \varphi \cap f^a$.
\end{theorem}

A result on the numbers of closed geodesics is given as a special case of this theorem. Given a closed manifold $M$ we let $\Lambda^1 M:= \Lambda_0 M \cap H^1(S^1,M)$. Since $\Lambda^1 M$ is a connected component of $H^1(S^1,M)$, it inherits a Hilbert manifold structure. Moreover, it is well-known that $\Lambda^1M$ and $\Lambda_0M$ are homotopy-equivalent.

\begin{definition}
Let $M$ be a closed manifold and $F:TM \to [0,+\infty)$ be a Finsler metric. We let 
$$E_F: \Lambda^1 M \to \RR, \quad E_F(\gamma) = \int_0^1F_{\gamma(t)}(\dot\gamma(t))^2\; dt,$$ denote its energy functional.  Given $a \in \RR$ we further let $N_S(F,a)$ denote the number of $SO(2)$-orbits of non-constant closed geodesics in $E_F^a$. If $F$ is reversible, then we let $N(F,a)$ denote the number of $O(2)$-orbits of non-constant closed geodesics in $E_F^a$. 
\end{definition}

It is well-known, see e.g. \cite{Mercuri}, that $E_F$ is of class $C^{1,1}$, satisfies the Palais-Smale condition and that its critical points are precisely the closed geodesics of $F$. From these observations one derives that Theorem \ref{TheoremLS} is applicable to $f=E_F$ and $\varphi$ being the time-1 map of a negative gradient flow of $E_F$. One obtains the following from Theorem \ref{TheoremLS}:

\begin{theorem}[{\cite[Theorem 3.11]{SpherComp}}]
\label{TheoremNumberOrbits}
Let $M$ be a closed manifold, $F:TM \to \RR$ be a Finsler metric, $E_F:\Lambda^1M \to [0,+\infty)$ be its energy functional and $a \in \RR$. Then 
$$N_S(F,a) \geq \SC_M(E_F^a)-1.$$
If $F$ is reversible, then 
$$N(F,a) \geq \SC_M(E_F^a)-1.$$ 
\end{theorem}

Thus, to ensure the existence of a certain number of closed geodesics in a sublevel set of $E_F$ it suffices to find corresponding lower bounds on the number $\SC_M(E_F^a)$ which we will do using the cohomology rings of $\Lambda_0M$. Throughout this note, we will always let $H^*$ denote singular cohomology.

\begin{definition}[{\cite{FarberGrantWeights}}]
Let $p:E \to B$ be a fibration, $A$ be an abelian group and let $u\in H^*(B;A)$ with $u \neq 0$. The \emph{sectional category weight of $u$}, denoted by $\wgt_p(u)$, is the largest $k \in \NN_0$, such that $f^*u=0$ for all continuous maps $f:X \to B$ with $\secat(f^*p)\leq k$, where $X$ is any topological space.
\end{definition}

The key observation about the sectional category weight of a fibration $p$ is the following, see \cite{FarberGrantWeights} for a proof: \\

\emph{ If there exists $u \in H^*(B;A)$ with $\wgt_p(u)\geq k$, then $\secat(p) \geq k+1$.}\\

We want to apply this to the fibration $r:BX \to \Lambda_0X$ and denote its weight function by
$$\wgt_1:=\wgt_r: H^*(\Lambda_0 M;R) \to \NN_0 \cup \{+\infty\}.$$
The following is a special case of \cite[Theorem 4.2]{SpherComp}:
\begin{prop}
\label{Propwgtprops}
Let $A$ be an abelian group and let $u \in H^*(\Lambda_0X;A)$ with $u\neq 0$.
\begin{enumerate}[a)]
\item $\wgt_1(u)\geq 1$ if and only if $c_1^*u=0 \in H^*(X;A)$.
\item If $A$ is a commutative ring and $v \in H^*(\Lambda_0X;A)$ with $u\cup v \neq 0$, then 
$$\wgt_1(u \cup v) \geq \wgt_1(u)+\wgt_1(v).$$
\item If $f:Y \to \Lambda_0X$ is continuous and $f^*u\neq 0$, then $\wgt_{f^*r}(f^*u)\geq \wgt_1(u)$.
\item Let $r_0$ be the zero map, $r_1:= r$ and let $r_n:B_nX \to \Lambda_0X$ be the $n$-fold fiberwise join of $r$ with itself for each $n \geq 2$. Then 
$$\wgt_1(u) = \sup\{ n \in \NN \ | \ r_n^*u=0\}.$$
\end{enumerate}
\end{prop}

Part b) of Proposition \ref{Propwgtprops} shows that non-vanishing cup products can produce classes of weight two or bigger. In addition, the following lemma provides another criterion, that does not involve the cup product, for cohomology classes to have weight two or bigger.  In the following, let $\Lambda^2X := C^0(S^2,X)$.
 
\begin{lemma}
\label{LemmaWeightMV}
Let $X$ be a Hausdorff space, $A$ be an abelian group, $k \geq 2$ and $u \in H^k(\Lambda_0X;A)$ with $u \neq 0$. If $u \in \ker c_1^*$ and $H^{k-1}(\Lambda^{2}X;A)=0$, then $\wgt_1(u)\geq 2$.
\end{lemma} 
\begin{proof}
By part d) of Proposition \ref{Propwgtprops} it suffices to show that $r_{2}^*u=0$ under the given assumptions, where $r_{2}:B_{2}X \to \Lambda_0X$ denotes the fiberwise join of $r$ with itself. We have seen in \cite[Section 4.4]{SpherComp} for a general fibration that this fiberwise join admits a Mayer-Vietoris sequence, which in our setting takes the form 
$$\xymatrix{
\dots \ar[r]  & H^{k-1}(\Lambda^{2}X;A) \ar[r]^{\delta} & H^k(B_{2}X;A) \ar[r] &  H^{k}(BX;A)\oplus H^{k}(BX;A) \ar[r] &  \dots\\
& & H^k(\Lambda_0X;A) \ar[u]_{r_2^*} \ar[ru]_{r^* \oplus r^*} & & 
}$$
Since $u \in \ker c_1^*$, and since, as shown in the proof of \cite[Theorem 1.13]{SpherComp}, there is a homotopy $r \simeq c_1 \circ e_1$, where $e_1:BX \to X$, $\gamma \mapsto \gamma(1)$, it holds that $u \in \ker r^*$. Hence, exactness of the top row yields $r_{2}^*u \in \im \delta$. But if $H^{k-1}(\Lambda^{2}X;A)=0$, then $\delta=0$ and the claim immediately follows.
\end{proof}

Throughout the following, we let $\iota_a: E_F^{<a} \hookrightarrow \Lambda_0M$
denote the inclusion of the open sublevel set for each $a \in \RR$. 

\begin{theorem}
\label{TheoremWeightCount}
Let $M$ be a closed manifold, $F:TM \to [0,+\infty)$ be a Finsler metric, $A$ be an abelian group and $a \in \RR$. If there exists $u \in H^*(\Lambda_0M;A)$ with $\wgt_1(u)\geq k$ and $\iota_a^*u \neq 0$, then  $$N_S(F,a-\varepsilon) \geq k,$$ where $\varepsilon >0$ is chosen so small that $(a-\varepsilon,a)$ contains only regular values of $E_F$. If $F$ is reversible, then $N(F,a-\varepsilon)\geq k$.
\end{theorem}
\begin{proof}
If $\iota_a^*u \neq 0$, then $\wgt_{\iota_a^*r}(\iota_a^*u) \geq \wgt_1(u)\geq k$ by part c) of Proposition \ref{Propwgtprops}. Hence, $\SC_M(E_F^{<a})\geq k+1$ by \eqref{EqSCrelsecat}. Since $E_F^{<a}$ can be deformed into $E_F^{a-\varepsilon}$ via a negative gradient flow, combining the deformation invariance of $\SC_M$ (see \cite[Section 2]{SpherComp}) with Theorem \ref{TheoremNumberOrbits} yields $N_S(F,a-\varepsilon) \geq \SC_M(E_F^{<a})-1 \geq k$ and analogously $N(F,a-\varepsilon)\geq k$ in the reversible case.
\end{proof}

\begin{cor}
\label{CorOrbits}
Let $M$ be a closed manifold, $F:TM \to [0,+\infty)$ be a Finsler metric, $A$ be an abelian group, $k\in \NN$ and $a > 0$. If  there exists $u \in H^*(\Lambda_0M;A)$ with $\wgt_1(u)\geq k$ and $\iota_a^*u\neq 0$, then $F$ admits $k$ distinct $SO(2)$-orbits of non-constant contractible closed geodesics of energy less than $a$. If $F$ is reversible, then it will admit $k$ distinct $O(2)$-orbits of non-constant contractible closed geodesics of energy less than $a$.
\end{cor} 
\begin{proof}
If the critical values of $E_F$ are non-isolated in $(-\infty,a)$, then the statement is obvious. If they are isolated, then pick $\varepsilon>0$ such that $(a-\varepsilon,a)$ does not contain critical values of $E_F$. The statement then follows from Theorem \ref{TheoremWeightCount}.
\end{proof}

The main significance of Corollary \ref{CorOrbits} is that one might employ it to detect distinct closed geodesics. The following definition makes this notion precise. 

\begin{definition}
Let $(M,F)$ be a Finsler manifold and let $\gamma_1,\gamma_2 \in \Lambda M$ be closed geodesics of $F$. 
\begin{enumerate}
\item We call $\gamma_1$ and $\gamma_2$ \emph{geometrically distinct} if $\gamma_1(S^1) \neq \gamma_2(S^1)\subset M$. 
\item We call $\gamma_1$ and $\gamma_2$ \emph{positively distinct} if they are either geometrically distinct or they lie in the same $O(2)$-orbit of $\Lambda M$, but not in the same $SO(2)$-orbit. 
\end{enumerate}
We will further let $\ell_F>0$ denote the length of the shortest non-constant closed geodesic of $F$.
\end{definition}

We want to show that under certain conditions on $M$ and $F$, the Finsler metric admits multiple \emph{positively distinct} closed geodesics and even \emph{geometrically distinct} ones if $F$ is reversible. We will do so using the following line of argument.\bigskip
 
\textbf{Strategy.} \enskip Let $M$ be a closed manifold and $F:TM \to [0,+\infty)$ be a Finsler metric that admits a contractible closed geodesic, e.g. for $M$ simply connected. 
\begin{enumerate}[1.]
\item Find an abelian group $A$, $k \geq 2$ and $u \in H^*(\Lambda_0M;A)$ with $\wgt_1(u)\geq k$. 
\item Find $r \in \NN$ and $a \in \RR$ so small that $E_F^{<a}$ does not contain any $m$-fold iterated closed geodesic for $m>r$ and show that $\iota_a^*u\neq 0$.
\end{enumerate}
Then there will be $\left\lceil \frac{k}{r} \right\rceil$ positively distinct closed geodesics in $E_F^{<a}$. If $F$ is reversible, then there will be $k$ geometrically distinct closed geodesics in $E^{<a}_F$. \\

\textit{Comment on the strategy.} \enskip Every closed geodesic $\gamma$ of $F$ satisfies $E_F(\gamma) \geq \ell_F^2$. Thus, for each $m \geq 2$, the $m$-fold iterate $\gamma^m$ of each closed geodesic satisfies
$$E_F(\gamma^m) = m^2E_F(\gamma)\geq m^2\ell_F^2.$$
Thus, if $a \leq (r+1)^2\ell_F^2$, then the sublevel set $E_F^{<a}$ can contain at most $r$-fold iterated closed geodesics. Consequently, if we can ensure that $E^{<a}_F$ contains $k$ distinct $SO(2)$-orbits, the existence of $\left\lceil \frac{k}{r} \right\rceil$ distinct closed geodesics immediately follows since each closed geodesics can contribute at most $r$ times to the count of the $SO(2)$-orbits.

\bigskip 
 Before we start to consider concrete examples, we want to conclude a section by presenting a result of Rademacher that we will use in all of the following examples.  In the above line of argument, we were searching for energy levels $a$ that satisfy inequalities of the form $a \leq r^2 \ell_F^2$ for suitable $r\in \NN$. To find such numbers, we will need lower bounds for the lengths $\ell_F$, which have been established by Rademacher for Finsler metrics of positive curvature. 
 
\begin{theorem}[{\cite[Theorem 1]{RadeSphere}, \cite[Theorem 4]{RadeSphere}}]
\label{TheoremRade}
Let $M$ be a closed simply-connected manifold and let $F:TM \to [0,+\infty)$ be a Finsler metric of reversibility $\lambda$ whose flag curvature satisfies $0<K \leq 1$. If one of the following holds:
\begin{enumerate}[a)]
\item $\dim M$ is even, 
\item $K > \frac{\lambda^2}{(1+\lambda)^2}$ and $\dim M \geq 3$,
\end{enumerate}
then $\ell_F \geq \pi\frac{1+\lambda}{\lambda}$.
\end{theorem}

\section{Even-dimensional spheres}
Let $n \in \NN$, let $g_1$ denote the round Riemannian metric on $S^n$ of constant curvature $1$ and let $E_1:\Lambda^1 S^n \to \RR$ denote its energy functional. For every $a \in \RR$ we let $\Lambda^{\leq a}S^n:= E_1^{-1}((-\infty,a])$ denote the corresponding closed sublevel set of $E_1$. 

The integral homology groups of free loop spaces of globally symmetric spaces can be computed using Morse-Bott theory of the energy functional, see \cite{ZillerFree}, which we want to discuss explicitly for $(S^n,g_1)$, see also \cite[Section 7]{OanceaGeod} for another discussion of this case. Using the results of \cite{ZillerFree}, we first give formulas for the cohomology groups of $\Lambda S^n$ with field coefficients.

\begin{prop}
\label{PropCohomSn}
Let $T_1S^n$ denote the unit tangent bundle of $S^n$ with respect to $g_1$ and let $\F$ be a field. Then for each $i \in \NN_0$ there are isomorphisms of $\F$-vector spaces
\begin{align}
H^i(\Lambda^{\leq (2\pi m)^2}S^n;\F) &\cong H^i(S^n;\F)\oplus \bigoplus_{k=1}^m H^{i-(2k-1)(n-1)}(T_1S^n;\F) \qquad \forall m \in \NN, \label{EqHSnsub} \\ 
H^i(\Lambda S^n;\F) &\cong H^i(S^n;\F)\oplus \bigoplus_{k=1}^\infty H^{i-(2k-1)(n-1)}(T_1S^n;\F). \label{EqHSn}
\end{align}
\end{prop}
\begin{proof}
It is proven in \cite{ZillerFree} that for each $i \in \NN_0$ there are group isomorphisms 
\begin{align}
H_i(\Lambda^{\leq (2\pi m)^2}S^n;\ZZ) &\cong H_i(S^n;\ZZ)\oplus \bigoplus_{k=1}^m H_{i-(2k-1)(n-1)}(T_1S^n;\ZZ) \qquad \forall m \in \NN,  \label{Eq1} \\ 
H_i(\Lambda S^n;\ZZ) &\cong H_i(S^n;\ZZ)\oplus \bigoplus_{k=1}^\infty H_{i-(2k-1)(n-1)}(T_1S^n;\ZZ). \label{Eq2}
\end{align}
Since $\F$ is a field, the universal coefficient theorem yields an isomorphism of $\F$-vector spaces $H_*(X;\F)\cong H_*(X;\ZZ) \otimes \F$ for every topological space $X$. Hence, tensoring both sides of the isomorphisms \eqref{Eq1} and \eqref{Eq2} with $\F$, we obtain that the analogous statements hold $\ZZ$ replaced by $\F$ and with respect to isomorphisms of $\F$-vector spaces. 

We further derive from the universal coefficient theorem in cohomology that 
$$H^*(X;\F) \cong \Hom_{\F}(H_*(X;\F),\F)$$ 
for any topological space $X$, so we obtain a dual isomorphism
\begin{align*}
H^i(\Lambda S^n;\F) &\cong \Hom_{\F}( H_i(\Lambda S^n;\F),\F) \\ 
&\cong \Hom_\F(H_i(S^n;\F),\F) \times  \prod_{k=1}^\infty \Hom_{\F}(H_{i-(2k-1)(n-1)}(T_1S^n;\F),\F)  \\
&\cong H^i(S^n;\F) \times \prod_{k=1}^\infty H^{i-(2k-1)(n-1)}(T_1S^n;\F)
\end{align*}
for each $i \in \NN_0$. Since $T_1S^n$ is a finite-dimensional manifold, $H^{i-(2k-1)(n-1)}(T_1S^n;\F)$ is finite-dimensional. Thus, it will be trivial for degree reasons for any fixed $i$ and sufficiently big $k \in \NN$, so the direct product in the above formula contains only finitely many non-trivial factors for each $i \in \NN_0$. Hence, it is isomorphic to the corresponding direct sum, yielding the isomorphism \eqref{EqHSn}. The isomorphisms \eqref{EqHSnsub} are derived along the same lines.
\end{proof}

It further follows from the Morse-theoretic construction underlying these isomorphisms that the maps $c_1^*:H^i(\Lambda S^n;\F)\to H^i(S^n;\F)$ and $j_{(2\pi m)^2}^*:H^i(\Lambda S^n;\F) \to H^i(\Lambda^{\leq (2\pi m)^2}S^n;\F)$ are the projections onto the corresponding summands under the isomorphisms in \eqref{EqHSnsub} and \eqref{EqHSn} for each $i \in \NN_0$, where $j_{(2\pi m)^2}: \Lambda^{\leq (2\pi m)^2}S^n \hookrightarrow \Lambda S^n$ denotes the inclusion for each $m \in \NN$.  \\

In the following, we consider even-dimensional spheres and odd-dimensional spheres separately, since the cohomology rings of $\Lambda S^n$ can be distinguished along these two cases.

\begin{theorem}
\label{TheoremGeodSpheresEven}
Let $n \geq 2$ and let $F:TS^{2n} \to [0,+\infty)$ be a Finsler metric on $S^{2n}$ of reversibility $\lambda$. If the flag curvature of $F$ satisfies $0 < K \leq 1$ and if
\begin{equation}
\label{EqAssumpEven}
 F_x(v) <\frac{1+\lambda}{\lambda} \sqrt{(g_1)_x(v,v)} \quad \forall (x,v) \in TS^{2n}, 
 \end{equation}
then $F$ will admit two positively distinct closed geodesics of length less than $2 \ell_F$. If $F$ is reversible, then the geodesics can be chosen geometrically distinct. 
\end{theorem}
\begin{proof}
As shown in \cite[Example 5.4.18]{Hausmann}, it holds that
$$H^i(T_1S^{2n};\ZZ_2) \cong \begin{cases}
\ZZ_2 & \text{if } i \in \{0,2n-1,2n,4n-1\}, \\
0 & \text{else.}
\end{cases}$$
It follows from this observation and from \eqref{EqHSnsub} and \eqref{EqHSn} that $H^{4n-2}(\Lambda S^{2n};\ZZ_2) \neq 0$ and
$$H^{4n-2}(\Lambda S^{2n};\ZZ_2) \cong H^{2n-1}(T_1S^{2n};\ZZ_2)\cong H^{4n-2}(\Lambda^{\leq 4\pi^2}S^{2n};\ZZ_2).$$
Let $x \in H^{4n-2}(\Lambda S^{2n};\ZZ_2) $ with $x\neq 0$. Since $H^{4n-2}(S^{2n};\ZZ_2)=0$, it holds that $x \in \ker c_1^*$, thus $\wgt_1(x)\geq 1$.  It further follows from the previous computation that $j_{4\pi^2}^*x \neq 0$. 

To apply Lemma \ref{LemmaWeightMV} to the class $x \in H^{4n-2}(\Lambda S^{2n};\ZZ_2)$, we note that the $\ZZ_2$-cohomology groups of $\Lambda^2S^{2n}=C^0(S^2,S^{2n})$ are well-known and it follows from \cite[Theorem 18.(1)]{Salvatore} and \cite[Corollary 10.26.4(b)]{NeisenBook} that there is an isomorphism of abelian groups
\begin{align*}
H^*(\Lambda^2 S^{2n};\ZZ_2) &\cong \ZZ_2[u_{4^i(2n-1)-1} \; | \; i \in \NN_0] \oplus \ZZ_2[u_{(4^i+1)(2n-1)} \;  |\; i \in \NN_0] \\
&\cong \ZZ_2[u_{2n-2},u_{8n-5},u_{32n-17},\dots] \oplus \ZZ_2[u_{4n-2},u_{10n-5},u_{34n-17},\dots],
\end{align*}
where the index of every class denotes its degree. This particularly shows that 
$$H^{4n-3}(\Lambda^2S^{2n};\ZZ_2)=0,$$ 
so that Lemma \ref{LemmaWeightMV} yields $\wgt_1(x)\geq 2$. By assumption \eqref{EqAssumpEven}, it holds for each $\gamma \in \Lambda^1 S^{2n}$ that
$$E_F(\gamma)= \int_0^1F_{\gamma(t)}(\dot\gamma(t))^2\; dt < \frac{(1+\lambda)^2}{\lambda^2}\int_0^1 (g_1)_{\gamma(t)}(\dot\gamma(t),\dot\gamma(t))\; dt = \frac{(1+\lambda)^2}{\lambda^2} E_1(\gamma).$$
By \cite[Theorem 4]{RadeSphere}, it follows from the flag curvature bounds of $F$ that
$\ell_F \geq \frac{1+\lambda}\lambda \pi,$
so the above chain of inequalities yields
$$E_F(\gamma) < \frac{\ell_F^2}{\pi^2}E_1(\gamma) \qquad \forall \gamma \in \Lambda^1 S^{2n}.$$
In particular, this shows that $\Lambda^{\leq 4 \pi^2}S^{2n} \subset E_F^{< 4\ell_F^2}$ and it follows that  $j_{4\pi^2}$ factors via the inclusion $\iota_{4\ell_F^2}:E_F^{<4\ell_F^2} \hookrightarrow \Lambda S^{2n}$, which yields $\iota_{4\ell_F^2}^*x\neq 0$. Since $\wgt_1(x)\geq 2$, it follows from Corollary \ref{CorOrbits} that $E^{<4\ell_F^2}_F$ contains two distinct $SO(2)$-orbits of non-constant closed geodesics and even two distinct $O(2)$-orbits if $F$ is reversible. 

Since the $m$-fold iterate, $m \geq 2$, of a non-constant closed geodesic $\gamma$ of $F$ satisfies
$$E_F(\gamma^m) = m^2 E_F(\gamma) \geq m^2\ell_F^2 \geq 4\ell_F^2,$$
the open sublevel set  $E^{<4\ell_F^2}_F$ contains only prime non-constant closed geodesics. Thus, since it contains two $SO(2)$-orbits or $O(2)$-orbits, resp., of non-constant closed geodesics, it must contain two positively distinct closed geodesics of $F$ and two geometrically distinct ones if $F$ is reversible, which shows the claim.
\end{proof}

Next we will prove the existence of two closed geodesics on even-dimensional spheres dropping assumption \eqref{EqAssumpEven} and imposing a much stronger pinching condition on the flag curvature of the Finsler metric instead.

\begin{theorem}
\label{TheoremGeodSpheresEvenPinch}
Let $n \geq 2$ and let $F:TS^{2n} \to [0,+\infty)$ be a Finsler metric on $S^{2n}$ of reversibility $\lambda$. Assume that the flag curvature $K$ of $F$ satisfies 
$$ \frac{9\lambda^2}{4(1+\lambda)^2} < K \leq 1.$$
Then $F$ admits two positively distinct closed geodesics whose lengths are less than $2\ell_F$. If $F$ is reversible, then the geodesics can be chosen geometrically distinct.
\end{theorem}

\begin{proof}
In the proof of Theorem \ref{TheoremGeodSpheresEven} we have constructed a cohomology class $x \in H^{4n-2}(\Lambda S^{2n};\ZZ_2)$ with $\wgt_1(x)\geq 2$.  It follows from standard arguments from the Morse theory of energy functionals, see e.g. \cite[Chapter 2]{Moore} or \cite{RadeNonrev}, that there is a non-constant closed geodesic $\gamma$ with $\ind(\gamma) \leq \deg x =4n-2 <3(2n-1)$, such that $i_{a}^*x\neq 0$, where $a :=E_F(\gamma)$ and where $i_{a}:E_F^a\hookrightarrow \Lambda S^{2n}$ denotes the inclusion. Let $\delta > \frac{9\lambda^2}{4(1+\lambda)^2}$ be given such that $K \geq \delta$. Then, applying the arguments of \cite[Remark 8.5]{RadeNonrev}, which generalizes \cite[(1.8)]{BTZclosed} from Riemannian to Finsler metrics, the pinching condition yields
$$a=E_F(\gamma) \leq\frac{9\pi^2}{\delta} < \frac{4(1+\lambda)^2}{\lambda^2}\pi^2 \leq 4\ell_F^2, $$
where we again used \cite[Theorem 4]{RadeSphere}. This particularly shows that $\iota_{4\ell_F^2}^*x\neq 0$, so by Corollary \ref{CorOrbits}, $E_F^{<4\ell_F^2}$ contains two distinct $SO(2)$-orbits of non-constant closed geodesics and two distinct $O(2)$-orbits if $F$ is reversible. Following the same line of argument as in the proof of Theorem \ref{TheoremGeodSpheresEven} shows the claim.
\end{proof}

\section{Odd-dimensional spheres}

For odd-dimensional spheres it is possible to show the existence of more than two closed geodesics, since the rational cohomology of their free loop spaces has a rich ring structure that we want to make use of. As a disadvantage, we need stronger assumptions on the flag curvature than in the even-dimensional case to obtain a lower bound on $\ell_F$, see Theorem \ref{TheoremRade}.


\begin{theorem}
\label{TheoremGeodSpheresOdd}
Let $n \in \NN$, let $F:TS^{2n+1}\to [0,+\infty)$ be a Finsler metric on $S^{2n+1}$, let $\lambda$ denote its reversibility and let $K$ denote its flag curvature. If 
\begin{equation}
\label{EqFg1est}
\frac{\lambda^2}{(1+\lambda)^2} < K \leq 1 \quad \text{and} \quad F_x(v) <  \frac{k+1}{2m}\cdot\frac{1+\lambda}{\lambda} \sqrt{(g_1)_x(v,v)} \quad \forall (x,v) \in TS^{2n+1}
\end{equation}
for some $k,m \in \NN$, then $F$ will admit $\left\lceil\frac{2m}{k} \right\rceil$ positively distinct closed geodesics of length less than $(k+1)\ell_F$. If $F$ is reversible, then the geodesics can be chosen geometrically distinct.
\end{theorem}
\begin{proof}
It was shown  in \cite{VPS}, see also \cite[Example 2.3.4]{Moore} or \cite[Section 5.4]{MenichiRHT}, that there is an isomorphism of rings
\begin{equation}
\label{EqSphereRingOdd}
H^*(\Lambda S^{2n+1};\QQ) \cong A(x) \otimes \QQ[y],
\end{equation}
where $\deg x=2n+1$ and $\deg y =2n$. Here, $A(x)$ denotes the exterior $\QQ$-algebra generated by $x$. Since $H^{2n}(S^{2n+1};\QQ)=0$, it holds that $y \in \ker c_1^*$, so parts a) and b) of Proposition \ref{Propwgtprops} show that 
$$\wgt_1(y^k)\geq k \qquad \forall k \in \NN.$$
We want to study the even powers $y^{2m} \in H^{4mn}(\Lambda S^{2n+1};\QQ)$, $m \in \NN$. We derive from the results of \cite{ZillerFree} and the universal coefficient theorem that
$$H^i(T_1S^{2n+1};\QQ) \cong \begin{cases}
\QQ & \text{if } i \in \{0,2n,2n+1,4n+1\}, \\
0 & \text{else.}
\end{cases}$$
From this computation and the isomorphisms \eqref{EqHSn} and \eqref{EqHSnsub} we obtain
$$H^{4mn}(\Lambda S^{2n+1};\QQ) \cong H^{2n}(T_1S^{2n+1};\QQ)\cong H^{4mn}(\Lambda^{\leq4m^2\pi^2} S^{2n+1};\QQ), $$
which yields $j_{4m^2\pi^2}^*(y^{2m})\neq 0$ for each $m \in \NN$. By the curvature condition, it follows from \cite[Theorem 1]{RadeSphere} that $\ell_F \geq \pi(1+\frac1\lambda)$. 
Using this estimate and the other assumption in \eqref{EqFg1est}, we compute for each $\gamma \in \Lambda^1S^{2n+1}$ that
$$E_F(\gamma)  < \frac{(k+1)^2}{4m^2}\frac{(1+\lambda)^2}{\lambda^2}E_1(\gamma) \leq \frac{(k+1)^2}{4m^2\pi^2}\ell_F^2 \cdot E_1(\gamma) \qquad \forall \gamma \in \Lambda^1M. $$
This particularly implies that $\Lambda^{\leq 4m^2\pi^2}S^{2n+1} \subset E_F^{<(k+1)^2\ell_F^2}$, which shows $\iota_{(k+1)^2\ell_F^2}^*(y^{2m})\neq 0$. Since $\wgt_1(y^{2m})\geq 2m$, it follows from Corollary \ref{CorOrbits} that $E_F^{<(k+1)^2\ell_F^2}$ contains $2m$ distinct $SO(2)$-orbits of closed geodesics of $F$ and $2m$ distinct $O(2)$-orbits if $F$ is reversible.
By the strategy outlined in section 2, this shows that $E^{<(k+1)^2\ell_F^2}_F$ contains $\lceil \frac{2m}{k} \rceil$ positively distinct closed geodesics of $F$ and $\lceil \frac{2m}{k} \rceil$ geometrically distinct ones if $F$ is reversible. 
\end{proof}

We want to state the case $k=1$ of Theorem \ref{TheoremGeodSpheresOdd} explicitly. 
\begin{cor}
Let $n \in \NN$, let $F:TS^{2n+1}\to [0,+\infty)$ be a Finsler metric on $S^{2n+1}$, let $\lambda$ denote its reversibility and let $K$ denote its flag curvature. If 
$$\frac{\lambda^2}{(1+\lambda)^2} < K \leq 1 \quad \text{and} \quad F_x(v) < \frac{1+\lambda}{m\lambda} \sqrt{(g_1)_x(v,v)} \quad \forall (x,v) \in TS^{2n+1}$$
for some $m \in \NN$, then $F$ will admit $2m$ positively distinct closed geodesics of length less than $2\ell_F$. If $F$ is reversible, then the geodesics can be chosen geometrically distinct.
\end{cor}

We next want to derive an analogue of Theorem \ref{TheoremGeodSpheresEvenPinch} for odd-dimensional spheres, in which we drop the assumed inequality between $F$ and the round metric from Theorem \ref{TheoremGeodSpheresOdd} and replace it by a stronger pinching condition. 

\begin{theorem}
\label{TheoremGeodSpheresOddPinch}
Let $n \in \NN$, let $F:TS^{2n+1}\to [0,+\infty)$ be a Finsler metric on $S^{2n+1}$, let $\lambda$ denote its reversibility and let $K$ denote its flag curvature. If 
\begin{equation*}
\frac{(2m+1)^2}{(k+1)^2}\frac{\lambda^2}{(1+\lambda)^2} <  K \leq 1,
\end{equation*}
for some $k,m \in \NN$, then $F$ will admit $\left\lceil \frac{2m}{k}\right\rceil$ positively distinct closed geodesics of length less than $(k+1) \ell_F$. If $F$ is reversible, then the geodesics can be chosen geometrically distinct.  
\end{theorem}
\begin{proof}
We have observed in the proof of Theorem \ref{TheoremGeodSpheresOdd} that for each $m \in \NN$ there exists $u_m \in H^{4mn}(\Lambda S^{2n+1};\QQ)$ with $\wgt_1(u_m)\geq 2m$. The critical point theory of energy functionals shows that there is a non-constant closed geodesic $\gamma_m$ with 
$$\ind(\gamma_m) \leq \deg u_m =4mn <(2m+1)(\dim S^{2n+1}-1)$$ 
and $i_{a_m}^*(u_m)\neq 0$, where $a_m :=E_F(\gamma_m)$ and $i_{a_m}:E_F^{a_m}\hookrightarrow \Lambda S^{2n+1}$ denotes the inclusion. As in the even-dimensional case, we apply \cite[Remark 8.5]{RadeNonrev}. If $\delta>\frac{(2m+1)^2}{(k+1)^2}\frac{\lambda^2}{(1+\lambda)^2}$ is chosen such that $K\geq \delta$, then
$$a_m \leq \frac{(2m+1)^2\pi^2}{\delta} < (k+1)^2\pi^2\frac{(1+\lambda)^2}{\lambda^2} \leq (k+1)^2 \ell_F^2$$
again by the estimate \cite[Theorem 1]{RadeSphere} on $\ell_F$. In particular, this shows that $\iota_{(k+1)^2\ell_F^2}^*u_m\neq 0$ and the claim follows in analogy with Theorem \ref{TheoremGeodSpheresOdd}.
\end{proof}

\begin{cor}
\label{CorOdd}
Let $n \in \NN$ and let $F:TS^{2n+1}\to [0,+\infty)$ be a Finsler metric on $S^{2n+1}$ of reversibility $\lambda$ whose flag curvature satisfies
$$\frac{9\lambda^2}{4(1+\lambda)^2}<K \leq 1.$$
Then $F$ admits two positively distinct closed geodesics of length less than $2\ell_F$. If $F$ is reversible, then the geodesics can be chosen geometrically distinct.
\end{cor}
\begin{proof}
This is the case $k=m=1$ of Theorem \ref{TheoremGeodSpheresOddPinch}.
\end{proof}

\begin{remark}
The existence part of Corollary \ref{CorOdd} is a consequence of a much stronger result by W. Wang from \cite{WangMultiple}, who has shown the existence of $n+1$ distinct closed geodesics under the weaker pinching condition $\frac{\lambda^2}{(1+\lambda)^2}<K \leq 1$. However, the methods used by Wang are entirely different than the ones used in this note and it does not follows from Wang's results that the length of the second-longest closed geodesic  of $F$ is less than $2\ell_F$.
\end{remark}

\section{Complex and quaternionic projective spaces}

Let $n \in \NN$ and let $\gFS$ be the Fubini-Study metric on $\CP^n$. Put $g_1 := 4\gFS$, let $E_1:\Lambda^1\CP^n \to \RR$ denote its energy functional and let $\Lambda^{\leq a}\CP^n:=E_1^{-1}((-\infty,a])$ denote its closed sublevel sets. Then the sectional curvature of $g_1$ satisfies $\frac14\leq K \leq 1$ and the results from \cite{ZillerFree} imply the following isomorphisms, that are derived in analogy with those from Proposition \ref{PropCohomSn}, for any field $\F$ and $i \in \NN_0$:
\begin{align}
&H^i(\Lambda^{\leq (2m\pi)^2}\CP^n;\F) \cong H^i(\CP^n;\F)\oplus \bigoplus_{k=1}^m H^{i-2(k-1)n-1}(T_1\CP^n;\F) \quad \forall m \in \NN, \label{EqHCPnsub} \\ 
&H^i(\Lambda \CP^n;\F) \cong H^i(\CP^n;\F)\oplus \bigoplus_{k=1}^\infty H^{i-2(k-1)n-1}(T_1\CP^n;\F), \label{EqHCPn}
\end{align}
where $T_1\CP^n$ denotes the unit tangent bundle of $\CP^n$ with respect to $g_1$. Furthermore, the inclusions $j_{(2\pi m)^2}:\Lambda^{\leq (2m\pi)^2}\CP^n \hookrightarrow \Lambda\CP^n$ and $c_1:\CP^n \hookrightarrow \Lambda\CP^n$ have the analogous properties to those described for spheres in the previous section. \\

Given $n \in \NN$, we let $g_1$ be the metric on $\HP^n$ for which $(\HP^n,g_1)$ is a globally symmetric space whose prime closed geodesics have length $2\pi$. It again follows from the results of \cite{ZillerFree} and is proven along the lines of Proposition \ref{PropCohomSn} that its free loop space has the following cohomology groups for a field $\F$ and each $i \in \NN_0$:
$$H^i(\Lambda\HP^n;\F) \cong H^i(\HP^n;\F) \oplus \bigoplus_{k=1}^\infty H^{i-2(k-1)(n+1)-3}(T_1\HP^n;\F). $$
Furthermore, the analogous formula to \eqref{EqHCPnsub} holds for the sublevel sets of the energy functional of $g_1$.

\begin{theorem}
\label{TheoremKPn}
Let $\mathbb{K} \in \{\CC,\mathbb{H}\}$ and $n \geq 3$. Let $F: T\KP^n \to [0,+\infty)$ be a Finsler metric on $\KP^n$ of reversibility $\lambda$. If the flag curvature of $F$ satisfies $0 < K \leq 1$ and if
$$F_x(v) < \frac{1+\lambda}{\lambda} \sqrt{(g_1)_x(v,v)} \qquad \forall (x,v) \in T\KP^n,$$
then $F$ will admit two positively distinct closed geodesics of length less than $2 \ell_F$. If $F$ is reversible, then the geodesics can be chosen geometrically distinct.
\end{theorem}

\begin{proof}
Consider the case $\K=\CC$. In the proof of \cite[Theorem 7.4]{SpherComp} applied to $M=\CP^n$, it was shown that there exists $u \in H^3(\Lambda \CP^n;\QQ)$ with $\wgt_1(u)\geq 2$. Using the results of \cite{ZillerFree} and the universal coefficient theorem, one obtains
$$H^i(T_1\CP^n;\QQ) \cong \begin{cases}
 \QQ & \text{if } i \in  \{2k \ | \ 0 \leq k \leq n-1\}\cup \{2k+2n-1 \ | \ 1 \leq k \leq n \}, \\
 0 & \text{else.}
\end{cases}$$ 
In terms of the isomorphisms in \eqref{EqHCPnsub} and \eqref{EqHCPn} this shows that
$$H^3(\Lambda \CP^n;\QQ) \cong H^2(T_1\CP^n;\QQ) \cong H^3(\Lambda^{\leq 4\pi^2}\CP^n;\QQ),$$
which yields $j_{4\pi^2}^*u\neq 0$. The remainder of the proof is carried out in analogy with the proof of Theorem \ref{TheoremGeodSpheresEven}.

In the case $\K=\mathbb{H}$, we can apply the same line of argument and obtain a class $u \in H^7(\Lambda\HP^n;\QQ)$ with $\wgt_1(u)\geq 2$. In analogy with the other case, one shows that $j_{4\pi^2}^*u\neq 0$ using the cohomology of $T_1\HP^n$ that was computed in \cite{ZillerFree} and concludes the proof in the same way as for $\K=\CC$.
\end{proof}

 \bibliography{SCgeod}

\providecommand{\bysame}{\leavevmode\hbox to3em{\hrulefill}\thinspace}
\providecommand{\MR}{\relax\ifhmode\unskip\space\fi MR }
\providecommand{\MRhref}[2]{%
  \href{http://www.ams.org/mathscinet-getitem?mr=#1}{#2}
}
\providecommand{\href}[2]{#2}
\begin{thebibliography}{CLOT03}

\bibitem[BTZ82]{BTZclosed}
Werner Ballmann, Gudlaugur Thorbergsson, and Wolfgang Ziller, \emph{Closed
  geodesics on positively curved manifolds}, Ann. of Math. (2) \textbf{116}
  (1982), no.~2, 213--247.

\bibitem[CLOT03]{CLOT}
Octav Cornea, Gregory Lupton, John Oprea, and Daniel Tanr\'e,
  \emph{Lusternik-{S}chnirelmann category}, Mathematical Surveys and
  Monographs, vol. 103, American Mathematical Society, Providence, RI, 2003.

\bibitem[DL10]{DuanLong}
Huagui Duan and Yiming Long, \emph{The index growth and multiplicity of closed
  geodesics}, J. Funct. Anal. \textbf{259} (2010), no.~7, 1850--1913.

\bibitem[Dua15]{DuanNonhyp}
Huagui Duan, \emph{Non-hyperbolic closed geodesics on positively curved
  {F}insler spheres}, J. Funct. Anal. \textbf{269} (2015), no.~11, 3645--3662.

\bibitem[FG08]{FarberGrantWeights}
Michael Farber and Mark Grant, \emph{Robot motion planning, weights of
  cohomology classes, and cohomology operations}, Proc. Amer. Math. Soc.
  \textbf{136} (2008), no.~9, 3339--3349.

\bibitem[Hau14]{Hausmann}
Jean-Claude Hausmann, \emph{Mod two homology and cohomology}, Universitext,
  Springer, Cham, 2014.

\bibitem[LD09]{LongDuan}
Yiming Long and Huagui Duan, \emph{Multiple closed geodesics on 3-spheres},
  Adv. Math. \textbf{221} (2009), no.~6, 1757--1803.

\bibitem[Men15]{MenichiRHT}
Luc Menichi, \emph{Rational homotopy---{S}ullivan models}, Free loop spaces in
  geometry and topology, IRMA Lect. Math. Theor. Phys., vol.~24, Eur. Math.
  Soc., Z\"{u}rich, 2015, pp.~111--136.

\bibitem[Mer77]{Mercuri}
Francesco Mercuri, \emph{The critical points theory for the closed geodesics
  problem}, Math. Z. \textbf{156} (1977), no.~3, 231--245.

\bibitem[Mes19]{SpherComp}
Stephan Mescher, \emph{Spherical complexities, with applications to closed
  geodesics}, to appear in Algebr. Geom. Topol., arXiv:1911.03948, 2019.

\bibitem[Moo17]{Moore}
John~Douglas Moore, \emph{Introduction to global analysis}, Graduate Studies in
  Mathematics, vol. 187, American Mathematical Society, Providence, RI, 2017,
  Minimal surfaces in Riemannian manifolds.

\bibitem[Nei10]{NeisenBook}
Joseph Neisendorfer, \emph{Algebraic methods in unstable homotopy theory}, New
  Mathematical Monographs, vol.~12, Cambridge University Press, Cambridge,
  2010.

\bibitem[Oan15]{OanceaGeod}
Alexandru Oancea, \emph{Morse theory, closed geodesics, and the homology of
  free loop spaces}, Free loop spaces in geometry and topology, IRMA Lect.
  Math. Theor. Phys., vol.~24, Eur. Math. Soc., Z\"{u}rich, 2015, With an
  appendix by Umberto Hryniewicz, pp.~67--109.

\bibitem[Rad04a]{RadeNonrev}
Hans-Bert Rademacher, \emph{Nonreversible {F}insler metrics of positive flag
  curvature}, A sampler of {R}iemann-{F}insler geometry, Math. Sci. Res. Inst.
  Publ., vol.~50, Cambridge Univ. Press, Cambridge, 2004, pp.~261--302.

\bibitem[Rad04b]{RadeSphere}
\bysame, \emph{A sphere theorem for non-reversible {F}insler metrics}, Math.
  Ann. \textbf{328} (2004), no.~3, 373--387.

\bibitem[Rad07]{RadeExist}
\bysame, \emph{Existence of closed geodesics on positively curved {F}insler
  manifolds}, Ergodic Theory Dynam. Systems \textbf{27} (2007), no.~3,
  957--969.

\bibitem[Sal04]{Salvatore}
Paolo Salvatore, \emph{Configuration spaces on the sphere and higher loop
  spaces}, Math. Z. \textbf{248} (2004), no.~3, 527--540.

\bibitem[She01]{Shen}
Zhongmin Shen, \emph{Lectures on {F}insler geometry}, World Scientific
  Publishing Co., Singapore, 2001.

\bibitem[VPS76]{VPS}
Micheline Vigu\'{e}-Poirrier and Dennis Sullivan, \emph{The homology theory of
  the closed geodesic problem}, J. Differential Geometry \textbf{11} (1976),
  no.~4, 633--644.

\bibitem[Wan19]{WangMultiple}
Wei Wang, \emph{Multiple closed geodesics on positively curved {F}insler
  manifolds}, Adv. Nonlinear Stud. \textbf{19} (2019), no.~3, 495--518.

\bibitem[Zil77]{ZillerFree}
Wolfgang Ziller, \emph{The free loop space of globally symmetric spaces},
  Invent. Math. \textbf{41} (1977), no.~1, 1--22.

\end{thebibliography}
 \bibliographystyle{amsalpha}

\end{document}